\documentclass[11pt]{article}

\usepackage{amsmath}
\usepackage{amsthm}

\theoremstyle{plain}
\newtheorem*{theorem}{Theorem}
\newtheorem{property}{Property}

\theoremstyle{definition}
\newtheorem*{remark}{Remark}

\begin{document}

\title{A Sufficient Condition for Nilpotency in a Finite Group}

\author{Benjamin Baumslag \and James Wiegold}

\maketitle

\begin{abstract}
It is shown that finite groups in which the order of the product of every pair of elements of co-prime order is the product of the orders, is nilpotent.
\end{abstract}

It is well known that every finite nilpotent group $G$ satisfies the following property.

\begin{property}\label{property:a}
The product of any two elements of $G$ of co-prime orders $k$ and $m$ has order $km$.
\end{property}

It does not seem to be known that the converse of this result holds: certainly it does not appear in any standard text that we know. The purpose of this short note is to establish this converse. A suitable reference for all terminology and results used here is \cite{huppert}.

\begin{theorem}
Every finite group $G$ satisfying Property~\ref{property:a} is nilpotent.
\end{theorem}

\begin{proof}
Note first that Property~\ref{property:a} extends to more than two elements of mutually co-prime orders, in the obvious way. What we do is to show that every Sylow subgroup is normal, and this is enough to establish nilpotency. Suppose that $G$ satisfies Property~\ref{property:a}, let $p_1$, $p_2$, \dots, $p_r$ be the distinct primes dividing the order of $G$, and $S_i$ a Sylow $p_i$-subgroup, for each $i$. The first step is to establish the equality
\begin{equation}\label{eq:1}
G=S_1S_2\cdots S_r.
\end{equation}

We do this by counting the number of elements on the right-hand side. Suppose that $s_i$, $t_i$ are elements of $S_i$ for $i=1$, $2$, \dots, $r$, and suppose that we have an equality
\[
s_1s_2\cdots s_r=t_1t_2\cdots t_r.
\]
Then
\[
s_1s_2\cdots s_{r-1}=t_1t_2\cdots t_rs^{-1}_r.
\]

If $t_rs^{-1}_r$ is not the identity, then by Property~\ref{property:a} the element on the right-hand side has order divisible by $p_r$, whereas that on the left-hand side does not. This is a contradiction, and so $s_r=t_r$. In the same way $s_i=t_i$ for each $i$. Thus, a count of the number of elements in the product $S_1S_2\cdots S_r$ shows that there are as many as there are elements in $G$, and this establishes equality~\eqref{eq:1}. Consider now a conjugate $x$ of an arbitrary element of the Sylow subgroup $S_1$. Then by what we just saw, $x$ is a product $s_1s_2\cdots s_r$ with obvious notation. Since $x$ is of order a power of $p_1$, Property~\ref{property:a} tells us that $s_2=s_3=\cdots=s_r=1$ and so that $x$ is in $S_1$. Thus $S_1$ is normal, and the same goes for all the other $S_i$, so all $S_i$ are normal and $G$ is nilpotent, as required.
\end{proof}

\begin{remark}
We do not know whether every finite group $G$ has a set of Sylow subgroups, one for each prime dividing the order of $G$, such that equality \eqref{eq:1} holds. Certainly soluble groups do, because of Philip Hall's celebrated theorem (see \cite[p.~665]{huppert}) stating that every soluble group $G$ has a set of Sylow subgroups that permute in pairs. Since the Sylow subgroups generate $G$, the fact that Sylow subgroups permute in pairs easily yields equation~\eqref{eq:1}.

On the other hand, some insoluble groups have this property. As an easiest example, consider the alternating group $A_5$ on $1$, $2$, $3$, $4$, $5$. It has the following Sylow-$2$, Sylow-$3$, and Sylow-$5$ subgroups respectively: $P=\langle(1,2)(3,4),(1,3)(2,4)\rangle$, $Q=\langle(1,2,3)\rangle$, $R=\langle(1,2,3,4,5)\rangle$. Here $PQ$ is the alternating group on $1$, $2$, $3$, $4$, of order $12$; since $R$ is of order $5$, we have $A_5=PQR$. Note that the Sylow subgroups have to be chosen carefully: not every choice will do. Other small simple groups have the property under discussion, as M.F.~Newman has pointed out in correspondence.
\end{remark}

\textbf{Acknowledgment}. We thank Professor Alan Camina for suggesting this topic to us.

\end{document}